\documentclass{conm-p-l}
\usepackage{mathrsfs}
\usepackage{amsmath,amssymb}

\def\bbR{\mathrm{I\!R}}
\def\bbC{{\mathchoice {\setbox0=\hbox{$\displaystyle\mathrm{C}$}
\hbox{\hbox to0pt{\kern0.4\wd0\vrule height0.9\ht0\hss}\box0}} 
{\setbox0=\hbox{$\textstyle\mathrm{C}$}\hbox{\hbox 
to0pt{\kern0.4\wd0\vrule height0.9\ht0\hss}\box0}} 
{\setbox0=\hbox{$\scriptstyle\mathrm{C}$}\hbox{\hbox 
to0pt{\kern0.4\wd0\vrule height0.9\ht0\hss}\box0}} 
{\setbox0=\hbox{$\scriptscriptstyle\mathrm{C}$}\hbox{\hbox 
to0pt{\kern0.4\wd0\vrule height0.9\ht0\hss}\box0}}}}

\def\rto{\bbR\nh^2}
\def\rtr{\bbR\nh^3}
\def\rfo{\bbR\hn^4}

\def\eu{e_1\w}
\def\ed{e_2\w}
\def\et{e_3\w}
\def\eq{e\hn_4\w}

\newcommand{\sca}{\mathrm{s}}
\newcommand{\w}{^{\phantom i}}
\newcommand{\nnh}{\hskip-1pt}
\usepackage{color}

\newcommand{\ric}{\mathrm{r}}

\def\cwedge{\bigcirc\kern-1.07em\wedge\ }
\newcommand{\hyp}{\hskip.5pt\vbox
{\hbox{\vrule width2.5ptheight0.5ptdepth0pt}\vskip2pt}\hskip.5pt}

\def\hs{\hskip.7pt}
\def\hh{\hskip.4pt}
\def\nh{\hskip-.7pt}
\def\hn{\hskip-.4pt}
\def\w{^{\phantom i}}
\def\tzm{{T\hskip-2.9pt_z\w\hn M}}

\def\aj{\alpha}

\def\bj{\beta}
\def\cj{\gamma}
\def\dj{\delta}
\def\ja{\lambda}
\def\ly{\lambda}
\def\jb{\mu}
\def\jc{\nu}

\def\vg{\varGamma}

\def\p{p}

\newtheorem{thm}{{\bf Theorem}}[section]

\newtheorem{prop}[thm]{{\bf Proposition}}
\newtheorem{cor}[thm]{{\bf Corollary}}

\theoremstyle{definition}

\theoremstyle{remark}
\newtheorem{rem}[thm]{{\bf Remark}}

\numberwithin{equation}{section}

\begin{document}

\title[Al\-most-prod\-uct K\"ah\-ler surfaces]{Al\-most-prod\-uct K\"ah\-ler
surfaces}
\author{Andrzej Derdzinski$^{1}$ and JeongHyeong Park$^{2}$}
\address{$^{1}$\nnh\ Department\nnh\ of\nnh\ Mathematics,\nnh\ The\nnh\
Ohio\nnh\ State\nnh\ University,\nnh\ Columbus,\nnh\ OH\nnh\ 43210,\nnh\ USA}
\email{andrzej@math.ohio-state.edu}

\address{$^{2}$ Department of Mathematics, Sungkyunkwan University, Suwon,
 16419, Korea}
\email{parkj@skku.edu}

\subjclass[2020]{53B35}
\keywords{K\"ah\-ler surface; al\-most-prod\-uct structure;
hol\-o\-mor\-phic distribution; totally geodesic foliation; exterior
differential system}

\begin{abstract}
We study the case where the tangent bundle of a K\"ah\-ler surface splits 
or\-thog\-o\-nal\-ly into two in\-te\-gra\-ble com\-plex-line 
sub\-bun\-dles. This amounts to the presence of a unit-length closed
anti-self-dual $\,2$-form, leading in turn to a curvature condition.
Therefore, such a decomposition need not exist, even locally,
in contrast with general Riemannian four-man\-i\-folds which, under the
assumption of real-an\-a\-lyt\-ic\-i\-ty, were shown by Grant and Vickers 
\cite{grant-vickers} to always admit, locally, two mutually orthogonal
in\-te\-gra\-ble rank-two distributions. 
We exhibit local coordinates naturally
adapted to a decomposition as above in a K\"ah\-ler surface, 
which may be used for simple constructions of examples. 
We also provide a characterization of the K\"ah\-ler case within 
the wider class of al\-most-K\"ah\-ler surfaces arising from the 
coordinates just mentioned. The characterization involves a system of four
first-or\-der quasi\-lin\-e\-ar partial differential equations imposed
on four unknown functions of four variables and, using Car\-tan's test, 
we prove the system's local solvability. 
Finally, we show that, for such a decomposition in a K\"ah\-ler surface, 
one of the summands is hol\-o\-mor\-phic if and only if the other one
has totally geodesic leaves, and describe a general construction of
examples with this last property.
\end{abstract}

\maketitle
\renewcommand{\thethm}{\Alph{thm}}
\renewcommand{\theprop}{\Alph{prop}}
\section{Introduction}\label{in}
\setcounter{equation}{0}
\setcounter{thm}{0}
Grant and Vickers \cite{grant-vickers} showed that every real-an\-a\-lyt\-ic 
Riemannian four-man\-i\-fold admits, 
locally, a decomposition of its tangent bundle into an orthogonal direct sum
of two in\-te\-gra\-ble real-plane sub\-bun\-dles. 
As we point out in  Proposition~\ref{crvob}, the analogous assertion 
fails to hold, in general, for 
real-an\-a\-lyt\-ic K\"ah\-ler surfaces, if one in  
addition requires the two summands to be $\,J$-in\-var\-i\-ant (with
$\,J\,$ denoting, here and below, the al\-most-com\-plex structure). 

The goal of this paper is to provide a complete local description -- see
Theorem~\ref{lcstr} below -- of 
K\"ah\-ler surfaces $\,(M\nh,g)\,$ that do admit such a decomposition:
the case where there exist $\,\mathcal{H}\,$ and $\,\mathcal{V}\hs$ such 
that 
\begin{equation}\label{cls}
\begin{array}{l}
T\nh M\hs=\,\mathcal{H}\oplus\mathcal{V}\,\mathrm{\hs\ is\hs\ the\hs\ 
}\,g\hyp\mathrm{or\-thog\-o\-nal\hs\ direct\hs\ sum\hs\ of}\\
\mathrm{two\ in\-te\-gra\-ble\ com\-plex}\hyp\mathrm{line\ sub\-bun\-dles\ 
}\hs\mathcal{H}\hh\mathrm{\ and\ }\hs\mathcal{V}\nh.
\end{array}
\end{equation}
In the following two propositions, proved in Sect.\,\ref{ps}, 
$\,M\,$ as well as the leaves of 
$\,\mathcal{H}\,$ and $\,\mathcal{V}\nh$, being al\-most-com\-plex
manifolds, are canonically oriented.
\begin{prop}\label{kheqv}On an al\-most-K\"ah\-ler surface\/ $\,(M\nh,g)\,$ 
the decompositions\/ {\rm(\ref{cls})} stand in 
a one-to-one correspondence with unit-length closed
anti-self-dual\/ $\,2$-forms. 
The correspondence associates with a given decomposition\/ {\rm(\ref{cls})}
the\/ $\,2$-form\/ $\,\zeta\,$ such that\/ $\,\mathcal{H}\hs$ and\/
$\,\mathcal{V}\hs$ are\/ $\,\zeta$-or\-thog\-o\-nal, while\/ $\,\zeta\,$
restricted to each leaf of\/ $\,\mathcal{H}\nh$, or\/ $\,\mathcal{V}\nh$,
equals\/ $\,-\nnh1/\nh\sqrt{2\hs}\,$
or, respectively, $\,1/\nh\sqrt{2\hs}\,$ times the area form of the leaf.
\end{prop}
In a distant analogy to the results of Tod \cite{tod}, cf.\ also
\cite{gauduchon-moroianu}, (\ref{cls}) implies a
curvature condition: a pointwise inequality relating 
the lowest eigen\-value 
of the Weyl tensor acting on anti-self-dual bi\-vec\-tors to the scalar
curvature $\,\sca$.
\begin{prop}\label{crvob}At points of a K\"ah\-ler surface\/ $\,(M\nh,g)\,$ 
where, locally, a decomposition\/ {\rm(\ref{cls})} exists, one necessarily
has
\begin{equation}\label{min}
\mathrm{min}\,\hs\mathrm{spec}\,W^-\nnh\nh\le\,\sca/6\,\mathrm{\ and,\
consequently,\ }\,|\hs W^-\nh|\ge-\hh\sca/6.
\end{equation}
\end{prop}
As a result, (\ref{cls}) does not hold, even locally, 
in a self-dual K\"ah\-ler surface with negative scalar curvature,
an example of which is provided by the complex hyperbolic plane with the
standard metric. For more examples, see Remark~\ref{nodec}.

The most trivial examples of (\ref{cls}) for K\"ah\-ler surfaces $\,(M\nh,g)\,$
arise when
\begin{equation}\label{trv}
\begin{array}{l}
\mathcal{H}\hs\mathrm{\ and\ }\hs\mathcal{V}\hh\mathrm{\ are,\nh\ locally,\nh\ 
the\ factor\ distributions}\\
\mathrm{of\ a\ Riem\-ann\-i\-an}\hyp\mathrm{prod\-uct\ decomposition\ of\ }\,g.
\end{array}
\end{equation}
Due to closedness of the Ric\-ci form, (\ref{cls}) is satisfied by 
the Ric\-ci eigen\-dis\-tri\-bu\-tions $\,\mathcal{H}\,$ and
$\,\mathcal{V}\nh$, in
any K\"ah\-ler surface $\,(M\nh,g)\,$ in which
\begin{equation}\label{cei}
\mathrm{the\ Ric\-ci\ tensor\ \,Ric\,\ of\ }\,g\,\mathrm{\ has\ two\ distinct\
constant\ eigen\-val\-ues.}  
\end{equation}
Several authors studied 
consequences of (\ref{cei}) under 
additional assumptions.

Apostolov, Dr\v aghici and Moroianu
\cite[Theorem\,1.2]{apostolov-draghici-moroianu} showed that, in the 
{\it compact\/} case, (\ref{cei}) implies (\ref{trv}) unless 
the Ric\-ci eigen\-val\-ues are both negative (and, if they are, 
$\,M\,$ must be a minimal surface of general type with even positive signature
and ample canonical bundle). Apostolov {\it et al.} also pointed out
that, without the
compactness assumption, (\ref{cei}) does not imply (\ref{trv}), with
counterexamples having one Ric\-ci eigen\-val\-ue equal to zero.

Nihonyanagi,\nh\ Oguro and Sekigawa \cite{nihonyanagi-oguro-sekigawa},
assuming {\it homogeneity\/} of $\hs(M\nh,g)$,\nh\ 
proved that (\ref{cei}) leads to (\ref{trv})
unless $\,\mathrm{Ric}\,$ has a zero eigen\-val\-ue. The latter case does
occur, as shown
by Kowalski \cite[Theorem VI.3]{kowalski}, while Dr\v aghici \cite{draghici}
explained why Kowalski's example is essentially unique in this regard.

Here and below we often use local coordinates
\begin{equation}\label{rng}
\begin{array}{l}
x^1\nh,\dots,x^4\hs\mathrm{\ on\ a\ coordinate\ 
domain\ }\hs\,U\subseteq M\,\mathrm{\ with\ the\ index}\\
\mathrm{ranges\ }\,j,k,l\in\{1,2,3,4\},
\hskip6pt\ja,\jb,\jc\in\{1,2\},
\hskip6ptp,q,r\in\{3,4\}.
\end{array}
\end{equation}
Any smooth functions $\,\aj,\bj,\cj,\dj\,$ of $\,x^1\nh,\dots,x^4$ give rise 
to a metric $\,g\,$ with
\begin{equation}\label{met}
\begin{array}{l}
(g_{11}\w,\,g_{12}\w,\,g_{22}\w)
=(e^{-\bj}\nh\sec\aj,\,\tan\aj,\,e^\bj\nh\sec\aj),\\
(g_{33}\w,\,g_{34}\w,\,g_{44}\w)\hh
=\hs(e^{-\dj}\nh\sec\cj,\,\,\tan\cj,\,\,e^\dj\nh\sec\cj),\\
(g_{13}\w,\,\,g_{14}\w,\,\,g_{23}\w,\,\,g_{24}\w)\,\,
=\,\,(0,\,0,\,0,\,0)\mathrm{,\ \ \hh where}\\
\hs\aj,\cj:U\to\,(-\pi/2,\pi/2)\,\mathrm{\ and\ }\,\bj,\dj:U\hh\to\nh\,\bbR\hh.
\end{array}
\end{equation}
The $\,(1,1)\,$ tensor field defined, using (\ref{met}), by
\begin{equation}\label{jtt}
\left[\begin{matrix}
J^1_1&J^1_2\cr
J^2_1&J^2_2\cr
\end{matrix}\right]
=\left[\begin{matrix}
-\nh g_{12}\w&-\nh g_{22}\w\cr
g_{11}\w&g_{12}\w\end{matrix}\right]\nnh\nnh,\quad
\left[\begin{matrix}
J^3_3&J^3_4\cr
J^4_3&J^4_4\cr
\end{matrix}\right]
=\left[\begin{matrix}
-\nh g_{34}\w&-\nh g_{44}\w\cr
g_{33}\w&g_{34}\w\end{matrix}\right]\nnh\nnh,\quad J_\ja^p\nh=J_p^\ja\nh=0,
\end{equation}
is an al\-most-com\-plex structure, since (\ref{met}) gives
\begin{equation}\label{det}
\begin{array}{rl}
\mathrm{i)}&\mathrm{det}\hs[g_{\ja\jb}\w]\,=\,\mathrm{det}\hs[g_{pq}\w]\,
=\,1\mathrm{,\ with\hs\ this\hs\ obvious\hs\ consequence\hskip-2.7pt:}\\
\mathrm{ii)}&(g^{11}\nnh,g^{12}\nnh,g^{22}\nnh,g^{33}\nnh,g^{34}\nnh,g^{44})
=(g_{22}\w,-\hn g_{12}\w,\,g_{11}\w,\,g_{44}\w,-\hn g_{34}\w,\,g_{33}\w)\hh.
\end{array}
\end{equation}
\begin{thm}\label{lcstr}
If the functions\/ $\,\aj,\bj,\cj,\dj\,$  in\/ {\rm(\ref{met})} satisfy the
system
\begin{equation}\label{sys}
\begin{array}{rl}
\mathrm{a)}&
(-\nnh1)^p\nh\bj_p\w=\hs\aj_q\w e^{(-\nnh1)^p\hn\dj}\nh\sec\cj-\aj_p\w\tan\cj
\mathrm{\,\ \ whenever\ }\,\{p,q\}=\{3,4\},\\
\mathrm{b)}&
(-\nnh1)^\ja\nh\dj\nh_\ja\w=\,\cj_\jb\w e^{(-\nnh1)^\ja\nh\bj}\nh\sec\aj
-\cj_\ja\w\tan\aj
\mathrm{\hn\ \ whenever\ }\,\{\ja,\jb\}=\{1,2\}
\end{array}
\end{equation}
of four first-or\-der partial differential equations, 
where the subscripts 
denote partial derivatives, then\/ $\,g\,$ in\/ {\rm(\ref{met})} is
a K\"ah\-ler metric on $\,\,U\,$ with the com\-plex-struc\-ture tensor $\,J\,$ 
given by\/ {\rm(\ref{jtt})}, and\/ {\rm(\ref{cls})} holds, on\/ $\,\,U$
rather than\/ $\,M\nh$, for\/ $\,\mathcal{H}\,$ and\/ $\,\mathcal{V}\hs$ 
spanned by the first two and last two coordinate vector fields.

Conversely, for any K\"ah\-ler surface\/ $\,(M\nh,g)\,$ with\/
{\rm(\ref{cls})}, $\,g,J\nh,\mathcal{H}\,$ and\/ $\,\mathcal{V}\hs$ arise,
locally, in suitable local coordinates\/ $\,y^1\nh,\dots,y^4\nh$, from the
above construction applied to some functions\/ $\,\aj,\bj,\cj,\dj\,$
satisfying\/ {\rm(\ref{sys})}.
\end{thm}
We also derive Theorem~\ref{lcstr}, in Sect.\,\ref{pt}, from  
our next result, which addresses the meaning of the system (\ref{sys}).
\begin{thm}\label{meang}
Formulae\/ {\rm(\ref{met})} -- {\rm(\ref{jtt})} always define an
al\-most-K\"ah\-ler metric, in the usual sense of skew-sym\-me\-try and
closedness of\/ $\,\omega=g(J\cdot\hs,\,\cdot\,)$. The five 
conditions listed below are mutually equivalent.
\begin{itemize}
\item[\rm{(i)}]Equations\/ {\rm(\ref{sys})}.
\item[\rm{(ii)}]$\nabla\hn\omega=0\,$ for\/
$\,\omega=g(J\cdot\hs,\,\cdot\,)\,$ and the Le\-vi-Ci\-vi\-ta connection\/ 
$\,\nabla\nh$ of\/ $\,g$.
\item[\rm{(iii)}]$\nabla\nnh J=0$, that is, the K\"ah\-ler property of\/ $\,g$.
\item[\rm{(iv)}]In\-te\-gra\-bi\-li\-ty of the al\-most-com\-plex structure\/
$\,J\nh$.
\item[\rm{(v)}]$\mathcal{V}\nh$-hol\-o\-mor\-phic\-i\-ty of\/ $\,\aj+i\bj\,$
and\/ $\,\mathcal{H}\nh$-hol\-o\-mor\-phic\-i\-ty of\/ $\,\cj+i\dj$.
\end{itemize}
Furthermore, as we show in Sect.\/\,{\rm\ref{pt}}, 
\begin{itemize}
\item[\rm{(vi)}]the first part of\/ {\rm(v)} is equivalent to\/ 
{\rm(\ref{sys}-a)}, the second to\/ {\rm(\ref{sys}-b)}.
\end{itemize}  
\end{thm}
We say here that a function $\,\phi:M\to\bbC\,$ on an al\-most-com\-plex
manifold $\,M\,$ is $\,\mathcal{V}\nh${\it-hol\-o\-mor\-phic}, for an 
in\-te\-gra\-ble 
com\-plex-line sub\-bun\-dle $\,\mathcal{V}\hh$ of $\,T\nh M\nh$,
if the restriction of $\,\phi\,$ to every leaf of $\,\mathcal{V}\hh$ is a 
hol\-o\-mor\-phic function on the leaf. (Note that the leaves 
of $\,\mathcal{V}\hs$ are complex manifolds.) 
\begin{rem}\label{swtch}The assumptions in (\ref{met}) -- (\ref{jtt}) and,
consequently, their conclusions, remain unaffected when one switches 
$\,(1,2,\aj,\bj)\,$ with $\,(3,4,\cj,\dj)$. 
\end{rem}
An obvious question arises as to when 
$\,\mathcal{H}\,$ or $\,\mathcal{V}\hs$ in (\ref{cls}) is hol\-o\-mor\-phic.
Without loss of generality, we may consider this to be the case for
$\,\mathcal{V}\nh$. 
Note that hol\-o\-mor\-phic\-i\-ty implies in\-te\-gra\-bi\-li\-ty (the
complex dimension being $\,1$), and $\,\mathcal{H},\mathcal{V}\hs$ must both 
be hol\-o\-mor\-phic whenever they are, locally, the factor distributions
of a Riem\-ann\-i\-an-prod\-uct decomposition of $\,g$.
\begin{thm}\label{noltd}
For any K\"ah\-ler surface\/ $\,(M\nh,g)\,$ with\/ {\rm(\ref{cls})}, 
$\,\mathcal{V}\hs$ in\/ {\rm(\ref{cls})} is hol\-o\-mor\-phic 
as a com\-plex-line sub\-bun\-dle of\/ $\,T\nh M\,$ if and only if
\begin{itemize}
\item[{\rm($*$)}]$\mathcal{H}\,$ has totally geodesic leaves.
\end{itemize}
In the case of\/ $\,g,\mathcal{H},\mathcal{V}\hs$ arising from the 
construction of Theorem\/~{\rm\ref{lcstr}}, 
{\rm($*$)} means that\/
$\,\aj\,$ and\/ $\,\bj\,$ or, equivalently, 
$\,g_{11}\w,g_{12}\w,g_{22}\w$, depend only on\/ $\,x^1$
and\/ $\,x^2\nh$.

All these conclusions remain valid when\/ $\,\mathcal{H}\,$ and\/
$\,\mathcal{V}\hs$ are switched.
\end{thm}
\begin{cor}\label{prdct}
Whenever\/ {\rm(\ref{cls})} holds on a K\"ah\-ler surface\/ $\,(M\nh,g)$,
hol\-o\-mor\-phic\-i\-ty of both\/  
$\,\mathcal{H}\,$ and\/ 
$\,\mathcal{V}\hs$ is equivalent to their being parallel, that is,
to\/ {\rm(\ref{trv})}.
\end{cor}
We prove Theorem~\ref{noltd} and Corollary~\ref{prdct} in Section~\ref{pd},
where we also point out (see Remark~\ref{smple}) that the situation of
Theorem~\ref{noltd} has a local description in terms of Theorem~\ref{lcstr},
much simpler than the general case of (\ref{cls}). More precisely,
Remark~\ref{smple} provides {\it a general local construction of K\"ah\-ler 
surfaces\/ $\,(M\nh,g)\,$ which constitute hol\-o\-mor\-phic 
bundles over real surfaces such that the horizontal distribution is 
in\-te\-gra\-ble and the bundle projection is a Riemannian sub\-mer\-sion.}

The following result, proved in Appendix II, establishes the existence of
numerous solutions to the system (\ref{sys}). 
\begin{prop}\label{slvbl}Real-an\-a\-lyt\-ic
solutions\/ 
$\,(\aj,\bj,\cj,\dj)\,$ to\/ {\rm(\ref{sys})} exist on a neighborhood
\hbox{of
any\/ $\hs z\hn\in\nh\bbR\hn^4\nh$ and realize all first-or\-der initial data
satisfying  
{\rm(\ref{sys})}\nh\ at $\,z$.}
\end{prop}

\renewcommand{\thethm}{\thesection.\arabic{thm}}
\renewcommand{\theprop}{\thesection.\arabic{prop}}
\section{Preliminaries}\label{pr}
\setcounter{equation}{0}
All manifolds and mappings, including
tensor fields, are assumed to be smooth, and all manifolds -- to be connected.

In local coordinates $\,y^1\nh,\dots,y^4\nh$,
for functions $\,\phi\,$ and $\,\psi\nh$, obviously,
\begin{equation}\label{epr}
\begin{array}{l}
\mathrm{the\hn\ differential\hn\ }2\hyp\mathrm{form\hn\
}\,\phi\,dy^1\wedge dy^2\nh+\,\psi\,dy^3\wedge dy^4\mathrm{\hn\ is}\\
\mathrm{closed\hs\ if\hs\ and\hs\ only\hs\ if\hs\ 
}\,\partial_3\w\phi=\partial\hn_4\w\phi=\partial_1\w\psi
=\partial_2\w\psi=0,
\end{array}
\end{equation}
that is, if and only if
$\,\phi\,$ and $\,\psi\,$ depend only on $\,y^1\nh,y^2$ and, 
respectively, $\,y^3\nh,y^4\nnh$.

The Nijen\-huis tensor $\,N\hs$ of an al\-most-com\-plex structure
$\,J\,$ sends vector fields $\,v,w\,$ to the vector field 
$\,N(v,w)=[v,w]+J[Jv,w]+J[v,Jw]-[Jv,Jw]$, easily seen to coincide with
$\,[J\nabla_{\nnh\!v}\w\nnh J
-\nabla_{\!\!J\hn v}\w J]\hh w+[\nabla_{\!\!J\hn w}\w J
-J\nabla_{\nh\!w}\w J]\hh v\,$ for any tor\-sion-free connection
$\,\nabla\nnh$. The New\-land\-er-Ni\-ren\-berg theorem now 
yields this well-known observation:
\begin{equation}\label{aci}
\begin{array}{l}
\mathrm{an\hs\ al\-most}\hyp\hs\mathrm{com\-plex\hs\ structure\hs\
parallel\hs\ with\hs\ re}\hyp\\
\mathrm{spect\hn\ to\hn\ a\hn\ tor\-sion}\hyp\mathrm{free\hn\ connection\hn\
is\hn\ in\-te\-gra\-ble.}
\end{array}
\end{equation}
For a $\,2$-form $\,\zeta\,$ on a manifold, $\,\varTheta=d\zeta$, and 
\begin{equation}\label{dph}
\varTheta(v,w,w')\,\mathrm{\ equals\ }\,d_v\w[\zeta(w,w')]\,
-\,\zeta([v,w],w')\,+\,\mathrm{cycl},
\end{equation}
$+\,\mathrm{cycl}\,$ meaning here {\it summed cyclically over\/} $\,v,w,w'\nh$.
See, e.g.,  \cite[p.\,21]{besse}.

We will use Car\-tan's lemma \cite[p.\,18]{sternberg}: in a
vector space, given any vectors $\,w\nh_1\w,\dots,w\nh_k\w$ and
linearly independent vectors $\,v\nh_1\w,\dots,v\nh_k\w$, one has
\begin{equation}\label{crt}
\begin{array}{l}
v\nh_1\w\wedge w\nh_1\w+\ldots+v\nh_k\w\wedge w\nh_k\w=0\,\mathrm{\nh\ if\nh\
and\ only\nh\ if\nh\ }\,w\nh_i\w
=h_{i1}\w v\nh_1\w+\dots+h_{ik}\w v\nh_k\w\mathrm{\nh\ for}\\
\mathrm{all\hs\ \hs}\,i\mathrm{,\hs\ with\hs\ some\hs\
\hs}\,h_{ij}\w\,\mathrm{\ forming\hs\ a\hs\ symmetric\hs\ }\,k\times k\,\mathrm{\hs\ matrix\hskip-3pt:\
}\,h_{ij}\w=\,h_{ji}\w.
\end{array}
\end{equation}
We also need the following well-known fact -- see, e.g., 
\cite[formula (11.1)]{derdzinski-piccione-terek}: 
\begin{equation}\label{zed}
\begin{array}{l}
\mathrm{in\ dimension\ }\hs\,n\,\mathrm{\ any\ differential\
}\hs\,n\hyp\mathrm{form\
without\ zeros\ equals,\ lo}\hyp\\ 
\mathrm{cally,\nh\
}\,dx^1\nnh\wedge dx^2\nnh\wedge\ldots\wedge\hs dx^n\nh\mathrm{\ for\ 
some\ local\ coordinates\ }\,x^1\nh,\dots,x^n\nh.
\end{array}
\end{equation}
\begin{rem}\label{krint}Given a rank-two differential $\,2$-form
$\,\zeta=\xi^1\wedge\hs\xi^2$
on a Riem\-ann\-i\-an man\-i\-fold $\,(M\nh,g)\,$ of dimension $\,n\ge2$,
the distribution $\,\mathrm{Ker}\,\hs\zeta\,$
is in\-te\-gra\-ble if and only if 
$\,d\hh\zeta=\eta\wedge\zeta\,$ for some $\,1$-form
$\,\eta$, and such $\,\eta\,$ becomes unique if one assumes that
$\,\eta=g(v,\,\cdot\,)\,$ with a vector field $\,v\,$ having 
$\,\zeta(v,\,\cdot\,)=0$. In fact, the `if' part is obvious from (\ref{dph}).
Assuming in\-te\-gra\-bi\-li\-ty of $\,\mathrm{Ker}\,\hs\zeta\,$ and using
(\ref{dph}), for $\,\varTheta=d\zeta$, we get
$\,\varTheta(v,v'\nh,\,\cdot\,)=0\,$ whenever
$\,\zeta(v,\,\cdot\,)=\zeta(v'\nh,\,\cdot\,)=0$. Extending
$\,\xi^1\wedge\hs\xi^2$ locally to a local trivialization
$\,\xi^1\nh,\dots,\xi^n$ of $\,T^*\hskip-1.8ptM\,$ dual to a local frame 
$\,e\hn_1\w,\dots,e\hn_ n\w$, and expressing $\,\varTheta$ as a
functional combination of linearly independent threefold 
exterior products of the $\,1$-forms $\,\xi^i$, we now see that the
only possibly nonzero coefficients occur when $\,i>2\,$ for at most one
$\,i\,$ present in the product, which proves our claim.
\end{rem}
\begin{rem}\label{holsb}For any hol\-o\-mor\-phic function $\,f:M\to\bbC\,$
without critical points on a complex manifold $\,M\,$ of any complex dimension
$\,n\ge2$, the distribution $\,\mathcal{H}\,$ on $\,M\,$ tangent to 
the level hyper\-sur\-faces
of $\,f\,$ is hol\-o\-mor\-phic 
as a vector sub\-bun\-dle of $\,T\nh M\nh$. In fact, given 
hol\-o\-mor\-phic vector fields $\,w\nh_j\w$, $\,j=1,\dots,n$, on an open set
$\,\,U\subseteq M\nh$, trivializing $\,T\hn U\nnh$, and rearranged so that
$\,f\hskip-2.2pt_1\w\ne0\,$ everywhere in $\,\,U\nnh$, with
$\,f\hskip-2.2pt_j\w$ denoting the 
directional derivative of $\,f\hh$ along $\,w\nh_j\w$, the restriction of 
$\,\mathcal{H}\,$ to $\,\,U$ has the 
hol\-o\-mor\-phic local trivialization $\,v\nh_j\w$, $\,j=2,\dots,n$, given
by $\,v\nh_j\w=f\hskip-2.2pt_1\w w\nh_j\w-f\hskip-2.2pt_j\w w\nh_1\w$.
\end{rem}
\begin{rem}\label{sdtwf}
It is well known -- see, for instance, \cite[Lemma 7.1]{derdzinski-park-shin}
-- that the spaces $\,\Lambda\nnh^\pm$ of self-dual and anti-self-dual 
bi\-vec\-tors in an oriented Euclidean four-space $\,\mathcal{T}$
are themselves naturally oriented, and their length $\,\sqrt{2\,}$
positive orthogonal bases
are precisely the same as all the triples
\begin{equation}\label{bas}
\pm\hh\eu\wedge\ed+\et\wedge\eq\hh,\quad 
\pm\hh\eu\wedge\et+\eq\wedge\ed\hh,\quad 
\pm\hh\eu\wedge\eq+\ed\wedge\et,
\end{equation}
where $\,\eu,\ed,\et,\eq$ is a positive or\-tho\-nor\-mal basis of 
$\,\mathcal{T}\nnh$. In addition, $\,\eu,\ed,\et,\eq$ is uniquely determined by
(\ref{bas}) up to an overall sign change. Since, as pointed out in
\cite[formula (7.7)]{derdzinski-park-shin}, the mapping sending
$\,(\eu,\ed,\et,\eq)\,$ to (\ref{bas}) is e\-qui\-var\-i\-ant relative to 
a two-fold covering homomorphism 
$\,\mathrm{SO}(4)\to\mathrm{SO}(3)\times\mathrm{SO}(3)$, 
\begin{equation}\label{ral}
\pm(\eu,\ed,\et,\eq)\,\mathrm{\ depends\
real}\hyp\mathrm{an\-a\-lyt\-ic\-al\-ly\ on\ (\ref{bas}).}
\end{equation}
\end{rem}
\begin{rem}\label{graff}Given a sub\-space $\,V\nnh$ of a vector space
$\,W\nnh$, the set $\,\mathcal{Y}$ of all sub\-spaces $\,H\,$ of $\,W\hs$
with $\,W\nh=H\oplus\hs V\hs$ is well known -- see, e.g., 
\cite[p.\,109]{bryant-chern-gardner-goldschmidt-griffiths} --
to carry a natural 
structure of an af\-fine space. To be specific, $\,\mathcal{Y}\hs$ stands in a
one-to-one correspondence with the pre\-im\-age of the point
$\,\mathrm{Id}\hh_{W\nh/\hh V}\w\nnh$ under the sur\-ject\-ive linear operator
$\,\mathrm{Hom}\hh(W\nh/\hh V\nh,W)\to\mathrm{End}\hh(W\nh/\hh V)\,$ 
sending any $\,\Phi\,$ to the composition $\,\Pi\Phi\nh$, where
$\,\Pi:W\to W\nh/\hh V\nnh$ is the quotient projection. The correspondence
associates with $\,H\in\mathcal{Y}\hs$ the inverse of the restriction 
$\,\Pi:H\to W\nh/\hh V\nnh$. More explicitly, the vector space
$\,\mathcal{X}=\mathrm{Hom}\hh(W\nh/\hh V\nh,V)\,$ acts on $\,\mathcal{Y}\,$
simply transitively, so that $\,H+\Psi\nh$, for
$\,(H,\Psi)\in\mathcal{Y}\times\nnh\mathcal{X}$, equals the image of $\,H\,$
under $\,\mathrm{Id}\hh_W\w\nnh+\Psi\Pi$.
\end{rem}

\section{Proofs of Propositions~\ref{kheqv} and~\ref{crvob}}\label{ps}
\setcounter{equation}{0}
\begin{proof}[{\it Proof of Proposition\/~{\rm\ref{kheqv}}}] 
We use the metric to identify vectors with $\,1$-forms, and 
bi\-vec\-tors with both $\,2$-forms and skew-ad\-joint $\,(1,1)\,$
tensors. The ``purely 
algebraic'' version of Proposition~\ref{kheqv} is also true. It arises when
one deletes the words {\it in\-te\-gra\-ble}, in (\ref{cls}), and {\it
closed}, and uses the area forms of $\,\mathcal{H}\,$ and
$\,\mathcal{V}\hs$ rather than referring to ``leaves'' (which now need not
exist). In fact, for $\,\zeta\hn $ described in the final clause of 
the proposition, $\,\sqrt{2\hh}\zeta^\pm\nnh$ are the initial entries 
in (\ref{bas}), where $\,\zeta^-\nnh\nh$ stands for $\,\zeta$, and 
$\,\sqrt{2\hh}\zeta^+\nnh\nh$ for the K\"ah\-ler form 
$\,\omega=g(J\cdot\hs,\,\cdot\,)$,
while $\,\eu,\ed\,$ and $\,\et,\eq$ are fixed positive or\-tho\-nor\-mal local
trivializations of $\,\mathcal{H}\,$ and $\,\mathcal{V}\nh$.
Due to Remark~\ref{sdtwf}, 
$\,*\hh\zeta^\pm\nnh\hn=\pm\hh\zeta^\pm\nnh$ 
and $\,|\hh\zeta^\pm\nh|=1$, while $\,\zeta^\pm\nnh$ uniquely 
determine $\,\mathcal{H}\,$ and $\,\mathcal{V}\hs$ since, obviously,
\begin{equation}\label{unq}
\mathcal{H}=\mathrm{Ker}\,(\zeta^+\nnh\nnh+\hs\zeta^-\nh)\,\,\mathrm{\ and\
}\,\,\mathcal{V}\nh=\mathrm{Ker}\,(\zeta^+\nnh\nnh-\hs\zeta^-\nh),
\end{equation}
showing that the assignment sending (\ref{cls}) to $\,\zeta\,$ is injective.

To verify its surjectivity, we start with $\,\zeta\,$ as in the first part of 
the proposition, let $\,\zeta^-\nnh\nnh=\zeta\,$ and 
$\,\sqrt{2\hh}\zeta^+\nnh\nnh=\omega\,$ (the K\"ah\-ler form), and 
then extend $\,\sqrt{2\hh}\zeta^\pm\nnh$ to length $\,\sqrt{2\,}$ 
positive or\-thog\-o\-nal local trivializations of the oriented bundles of 
self-dual and anti-self-dual $\,2$-forms, cf.\ Remark~\ref{sdtwf}. Next, we 
again invoke Remark~\ref{sdtwf}, including (\ref{ral}), to express the latter
local trivializations as (\ref{bas}) for a smooth local or\-tho\-nor\-mal
frame $\,(\eu,\ed,\et,\eq)$. Finally, we
set $\,\mathcal{H}=\mathrm{span}\hh(\eu,\ed)\,$ and 
$\,\mathcal{V}\nh=\mathrm{span}\hh(\et,\eq)$, so that surjectivity follows. 

Proving Proposition~\ref{kheqv} is thus reduced to showing 
that closedness of $\,\zeta$ implies 
in\-te\-gra\-bi\-li\-ty of both distributions $\,\mathcal{H},\mathcal{V}\nh$,
and vice versa. The first implication is straightforward: due to
closedness of $\,\omega=\sqrt{2\hs}\zeta^+\nnh\nh$, if 
$\,\zeta=\zeta^-\nnh\nh$ is closed, so are both  
$\,\zeta^+\hskip-3pt\pm\zeta^-\nnh$, and, consequently,
$\,\mathcal{H}\,$ and $\,\mathcal{V}\,$ are in\-te\-gra\-ble in view of
(\ref{unq}) and (\ref{dph}).

For the converse implication, use Remark~\ref{krint} and (\ref{unq})
to write
\[
d\hh[\hh\zeta^+\hskip-3pt\pm\zeta^-\nnh]\,
=\,\eta^\pm\nnh\nnh\wedge[\hh\zeta^+\hskip-3pt\pm\zeta^-\nnh]\,
=\,\eta\wedge[\hh\zeta^+\hskip-3pt\pm\zeta^-\nnh]
\]
for $\,\zeta^\pm\nnh$ and $\,\eu,\ed,\et,\eq$ as above, where 
$\,\eta=\eta^+\hskip-3pt+\eta^-\nh$, and the $\,1$-form $\,\eta^+\nh$, 
or $\,\eta^-\nh$, treated as a vector field, is a functional combination of 
$\,\eu,\ed$ or, respectively, $\,\et,\eq$. Add\-ing/\nh sub\-tract\-ing, we
get $\,d\hs\omega=\eta\wedge\omega\,$ and $\,d\hh\zeta=\eta\wedge\zeta$. As
$\,\omega\,$ is closed and nondegenerate, $\,\eta=0$, which completes the proof.
\end{proof}
In a Riemannian manifold of dimension $\,n\ge4\,$ with the scalar curvature
$\,\sca$, the action of the Weyl tensor on $\,2$-forms $\,\zeta\,$
satisfies the following well-known formula
\cite[p.\,409]{derdzinski-83}, \cite[p.\,458]{derdzinski-00}, easily 
derived from the Ric\-ci identity:
\begin{equation}\label{wei}
W\nnh\zeta=\displaystyle{\frac12}\left[\delta(\nabla\nh\zeta-d\zeta)
-d\hs\delta\zeta\right]
+\displaystyle{\frac{n-4}{2(n-2)}}\,\{\ric,\zeta\}
+\displaystyle{\frac{\sca}{(n-1)(n-2)}}\hs\zeta,
\end{equation}
where $\,\{\,,\}\,$ is the anticommutator and $\,\ric\,$ the Ric\-ci tensor.
In local coordinates,
\begin{equation}\label{lco}
\begin{array}{l}
W_{\!ijpq}\w\zeta^{pq}\nh=-\,\zeta_{pi,\hh j}\w{}^p-\zeta_{jp,\hh i}\w{}^p
-\zeta_{pj,}\w{}^p{}_i\w+\zeta_{pi,}\w{}^p{}_j\w\phantom{\displaystyle{\frac12}}\\
\phantom{W_{\!ijpq}\w\zeta^{pq}\nnh}
+\,\displaystyle{\frac{n-4}{n-2}}(\ric_j^{\hs p}\zeta_{ip}\w
+\ric_i^{\hs p}\zeta_{pj}\w)
+\displaystyle{\frac{2\hskip.8pt\sca}{(n-1)(n-2)}}\hs\zeta\hh_{i\hn j}\w,
\end{array}
\end{equation}
the index range, in contrast with (\ref{rng}), being $\,i,j,p,q=1,\dots,n$, 
and $\,\delta(\nabla\nh\zeta-d\zeta)$ having the coordinate expression
$\,-\zeta_{pi,\hh j}\w{}^p-\zeta_{jp,\hh i}\w{}^p
=\zeta\hh_{i\hn j,\hh p}\w{}^p-(\zeta\hh_{i\hn j,\hh p}\w+\zeta_{pi,\hh j}\w
+\zeta_{jp,\hh i}\w)^{,\hh p}\nh$. As
$\,2\hh|\hh\zeta|^2\nh=\zeta\hh^{i\hn j}\nh\zeta\hh_{i\hn j}\w$,
the La\-plac\-i\-an
of $\,|\hh\zeta|^2$ equals
$\,\zeta\hh^{i\hn j,\hh p}\nh\zeta\hh_{i\hn j,\hh p}\w
+\zeta\hh^{i\hn j}\nh\zeta\hh_{i\hn j,\hh p}\w{}^p=2\hh|\nabla\nh\zeta|^2\nh
+2\langle\zeta,\delta\hh\nabla\nh\zeta\rangle$. 
When $\,|\hh\zeta|\,$ is constant,\nh\ $\,n=4$,\nh\ and
$\,d\zeta=\delta\zeta=0$,\nh\ (\ref{wei}) with
$\,\langle\zeta,\delta\hh\nabla\nh\zeta\rangle=-|\nabla\nh\zeta|^2\nh$ gives
\begin{equation}\label{wzz}
\langle W\nnh\zeta,\zeta\rangle\,-\,\sca\hs|\hh\zeta|^2\nnh/\nh6\,
=\,-|\nabla\nh\zeta|^2\nnh/\nh2.
\end{equation}
\begin{proof}[{\it Proof of Proposition\/~{\rm\ref{crvob}}}]
The existence of a decomposition (\ref{cls}) gives rise, via 
Proposition~\ref{kheqv}, to a unit-length closed anti-self-dual $\,2$-form
$\,\zeta\,$ which, due to its anti-self-dual\-i\-ty, is also co\hh-clos\-ed.
Now the first part of (\ref{min}) follows from (\ref{wzz})
since $\,\mathrm{min}\,\hs\mathrm{spec}\,W^-\nnh\nh$,
at any point $\,z$,
equals the minimum of $\,\langle W\nnh\zeta,\zeta\rangle\,$
over all unit-length anti-self-dual exterior $\,2$-forms 
$\,\zeta\,$ at $\,z$. The second part of (\ref{min}) is in turn  obvious as
trace\-less\-ness of $\,W^-\nnh$ gives 
$\,|\hs W^-\nh|\ge-\mathrm{min}\,\hs\mathrm{spec}\,W^-\nnh\nh$.
\end{proof}
\begin{rem}\label{nodec}Due to Propositions~\ref{kheqv}, and~\ref{crvob}, both 
(\ref{min}) and (\ref{cls}), fail to hold, even locally, 
in self-dual K\"ah\-ler surfaces of negative scalar curvature, 
such as the standard complex hyperbolic plane, some examples constructed in 
\cite{derdzinski-81}, and small deformations of the above metrics,
arising, for instance, when one adds to the K\"ah\-ler form 
$\,i\hskip1pt\partial\overline{\partial}\hh\varphi\,$ for a real-val\-ued 
function $\,\varphi\,$ which is $\,C^4\nh$-close to zero.  
\end{rem}

\section{Proofs of Theorems~\ref{meang} and~\ref{lcstr}}\label{pt}
\setcounter{equation}{0}
Before proving the theorems, we derive some consequences of
(\ref{met}) -- (\ref{jtt}), which will be used in Sect.\,\ref{pd} as well.
Without assuming (\ref{sys}), note that, by (\ref{met}),
\begin{equation}\label{acc}
g_{\ja\ja}\w=e^{(-\nnh1)^\ja\nh\bj}\nh\sec\aj,\quad g_{12}\w=\tan\aj,\quad
g_{pp}\w=e^{(-\nnh1)^p\hn\dj}\nh\sec\cj,\quad g_{34}\w=\tan\cj
\end{equation}
if $\,\ja\in\{1,2\}\,$ and $\,p\in\{3,4\}$. Therefore,
one has (\ref{sys}) if and only if
\begin{equation}\label{bpe}
\mathrm{a)}\hskip6pt
(-\nnh1)^p\nh\bj_p\w=g_{pp}\w\aj_q\w-g_{34}\w\aj_p\w\hh,\qquad
\mathrm{b)}\hskip6pt
(-\nnh1)^\ja\nh\dj_\ja\w=g_{\ja\ja}\w\cj_\jb\w-g_{12}\w\cj_\ja\w
\end{equation}
whenever $\,\{\ja,\jb\}=\{1,2\}\,$ and\/ $\,\{p,q\}=\{3,4\}$.
From (\ref{met}) -- (\ref{jtt}) we also get
\begin{equation}\label{jpq}
J_p^q\nh=\nh(-\nnh1)^q\nh e^{(-\nnh1)^p\hn\dj}\nh\sec\cj,\,\,\,J_p^p\nh
=\nh(-\nnh1)^p\nh\tan\cj\,\mathrm{\ (no\ summing)\ if\ }\{p,q\}\nh=\nh\{3,4\}.
\end{equation}
For $\,g,J\,$ defined by (\ref{met}) -- (\ref{jtt}) and
$\,\omega=g(J\cdot\hs,\,\cdot\,)$, one easily sees that
\begin{equation}\label{omg}
\begin{array}{l}
\omega_{12}\w\nh=\hs\omega_{34}\w\nh=\hs-\omega_{21}\w\nh=\hs-\omega_{43}\w\hn
=1\,\,\mathrm{\ and\ }\,\,\omega_{jk}\w\nh=0\,\,\mathrm{\ other}\hyp\\
\mathrm{wise,\ or,\ equivalently,\ }\,\omega\,=\,dx^1\wedge dx^2\hs
+\,dx^3\wedge dx^4\nh.
\end{array}
\end{equation}
Thus, in terms of the components $\,\vg_{\hskip-2.2ptjk}^l$ of 
the Le\-vi-Ci\-vi\-ta connection $\,\nabla\nh$ of the metric
$\,g$, with $\,k'\nh,l'$ defined by $\,(1'\nh,2'\nh,3'\nh,4')=(2,1,4,3)$, the
components of $\,\nabla\hn\omega\,$ are
\begin{equation}\label{nao}
\omega_{kl'\nnh,\,j}\w\hs
=\,(-\nnh1)^l\vg_{\hskip-2.2ptjk}^l\hs
+\,(-\nnh1)^k\vg_{\hskip-2.2ptjl'}^{\hs k'}\hskip2.7pt.
\end{equation}
On the other hand, due to skew-sym\-me\-try of $\,\omega$, cf.\ (\ref{omg}),
\begin{equation}\label{cpd}
\nabla\hn\omega\,\mathrm{\ is\ completely\ determined\ by\ the\ components\ 
}\,\omega_{kl,\,j}\w\hs\mathrm{\ with\ }\,k<l\hh.
\end{equation}
With the index ranges (\ref{rng}),
$\,\vg_{\hskip-2.2ptj\ja}^\ja\nh=\vg_{\hskip-2.2ptjp}^{\hs p}\nh=0\,$ due
to (\ref{det}-i). Thus,
\begin{equation}\label{oot}
\omega_{12,\,j}\w\hs=\,\omega_{34,\,j}\w\hs=\,0\hh,  
\end{equation}
since 
$\,\omega_{12,\,j}\w\nh=-\vg_{\hskip-2.2ptj\ja}^\ja$ and 
$\,\omega_{34,\,j}\w\nh=-\vg_{\hskip-2.2ptjp}^{\hs p}$ in view of 
(\ref{nao}).

Next, (\ref{nao}),  (\ref{met}), (\ref{det}-ii) and the
Chris\-tof\-fel-sym\-bol formula yield the following
equalities, valid whenever $\,\{\ja,\jb\}=\{1,2\}\,$ and 
$\,\{p,q\}=\{3,4\}$, with repeated indices 
{\it not\/} summed over, and (\ref{olp}-d) arising as a trivial
consequence of (\ref{acc}):
\begin{equation}\label{olp}
\begin{array}{rll}
\mathrm{a)}&\omega_{\ja p,\,\ja}\w\hs
=\,(-\nnh1)^q\vg_{\hskip-2.2pt\ja\ja}^q\hs
+\,(-\nnh1)^\ja\vg_{\hskip-2.2pt\ja p}^{\hs\jb},&
\omega_{\ja p,\,\jb}\w\hs
=\,(-\nnh1)^q\vg_{\hskip-2.2pt12}^q\hs
+\,(-\nnh1)^\ja\vg_{\hskip-2.2pt\jb p}^{\hs\jb},\\
\mathrm{b)}&2\vg_{\hskip-2.2pt\ja\ja}^q\hs
=\,g_{34}\w\partial\nh_p\w\hs g_{\ja\ja}\w\nh
-\hs g_{pp}\w\partial\nh_q\w\hs g_{\ja\ja}\w,&
2\vg_{\hskip-2.2pt\ja p}^{\hs\jb}\hs
=\,g_{\ja\ja}\w\partial\nh_p\w\hs g_{12}\w\nh
-\hs g_{12}\w\partial\nh_p\w\hs g_{\ja\ja}\w,\\
\mathrm{c)}&2\vg_{\hskip-2.2pt12}^q\hs
=\,g_{34}\w\partial\nh_p\w\hs g_{12}\w\nh
-\hs g_{pp}\w\partial\nh_q\w\hs g_{12}\w,&
2\vg_{\hskip-2.2pt\jb p}^{\hs\jb}\hs
=\,g_{\ja\ja}\w\partial\nh_p\w\hs g_{\jb\jb}\w\nh
-\hs g_{12}\w\partial\nh_p\w\hs g_{12}\w,\\
\mathrm{d)}&\partial\nh_p\w\hn\log g_{\ja\ja}\w\nnh\nh
=\nh(-\nnh1)^\ja\bj_p\w\nnh
+\nh g_{12}\w\aj_p\w,&
\partial\nh_p\w\hs g_{12}\w=g_{11}\w g_{22}\w\aj_p\w.
\end{array}
\end{equation}
If $\,\{\ja,\jb\}=\{1,2\}\,$ and\/ $\,\{p,q\}=\{3,4\}$, (\ref{omg})
and (\ref{olp}) give
\begin{equation}\label{tol}
\begin{array}{l}
2(-\nnh1)^\ja\omega_{\ja p,\,\ja}\w\hs/\nh g_{\ja\ja}\w
=[(-\nnh1)^{\ja+p}\nh g_{12}\w\hs+\,g_{34}\w]
\hh[g_{pp}\w\aj_q\w\hs-\,g_{34}\w\aj_p\w\hs-\,(-\nnh1)^p\bj_p\w]\\
\hskip28pt+\,g_{pp}\w\hh[g_{qq}\w\aj_p\w\hn-\nh g_{34}\w\aj_q\w\hn
-\nh(-\nnh1)^q\nh\bj_q\w]\nh
+\nh(g_{11}\w g_{22}\w\hn-\nh g_{12}^2\nh+g_{34}^2\nh
-\nh g_{33}\w g_{44}\w)\aj_p\w,\\
\hskip38pt\mathrm{where,\ due\ to\ (\ref{det}}\hyp\mathrm{i),\ 
}\,(g_{11}\w g_{22}\w\hs-\,g_{12}^2\nh+g_{34}^2\hs
-\,g_{33}\w g_{44}\w)\aj_p\w\hs=\,0,\\
2\hs\omega_{\ja p,\,\jb}\w\hs
=\,(-\nnh1)^pg_{11}\w g_{22}\w\hh[g_{pp}\w\aj_q\w\hs-\,g_{34}\w\aj_p\w\hs
-\,(-\nnh1)^p\bj_p\w].
\end{array}
\end{equation}
\begin{proof}[{\it Proof of Theorem\/~{\rm\ref{meang}}}]The al\-most-K\"ah\-ler claim is
obvious from (\ref{omg}), while (\ref{bpe}), (\ref{cpd}), (\ref{oot}) and
(\ref{tol}) establish the equivalence between (i) and (ii). As
$\,\omega=g(J\cdot\hs,\,\cdot\,)$, (ii) is in turn equivalent to (iii). Next,
(\ref{aci}) shows that (iii) implies (iv), while -- see
\cite[Prop.\,2.29]{besse} -- (iv) and closedness of 
$\,\omega$, cf.\ (\ref{omg}), give (iii).
Finally, the $\,\mathcal{V}\nh$-hol\-o\-mor\-phic\-i\-ty of $\,\aj+i\bj\,$
amounts to 
$\,(d\aj+i\hs d\bj)J=i(d\aj+i\hs d\bj)$, that is,
$\,(d\aj)J=-\hs d\bj\,$ which, by (\ref{jpq}), is nothing else than 
(\ref{sys}-a).

This proves the first assertion in (vi) and, 
combined with Remark~\ref{swtch}, yields the second one as well,
completing the proof.
\end{proof}
\begin{proof}[Proof of Theorem~\ref{lcstr}]The first claim is immediate
from the fact that (iii) follows from (i) in Theorem~\ref{meang}.

Let us now assume (\ref{cls}). The two sub\-bun\-dles are spanned,
locally, by the first two and last two coordinate vector fields for some 
local coordinates $\,y^1\nh,\dots,y^4\nh$.
Consequently, the K\"ah\-ler form
$\,\omega=g(J\cdot\hs,\,\cdot\,)\,$ equals
$\,\phi\,dy^1\wedge dy^2\nh+\,\psi\,dy^3\wedge dy^4$ for some functions
$\,\phi,\psi\,$ and, according to the line following (\ref{epr}), the relation 
$\,d\hs\omega=0$ shows that
$\,\phi\,$ and $\,\psi\,$ are functions of $\,y^1\nnh,y^2$ and 
$\,y^3\nh,y^4\nh$, respectively. Now (\ref{zed}) gives (\ref{omg}) 
with some new coordinates $\,x^1\nh,\dots,x^4\nh$. This yields (\ref{det}-i)
since, by
(\ref{omg}), $\,dx^1\wedge dx^2$ (or, $\,dx^3\wedge dx^4$) restricted to each 
leaf of $\,\mathcal{H}\,$ (or, $\,\mathcal{V}$) in (\ref{cls}) equals the
area form of the restriction of $\,g\,$ to the leaf.
Due to (\ref{det}-i) and $\,g$-or\-thog\-o\-nal\-i\-ty in (\ref{cls}), some 
unique functions $\,\aj,\bj,\cj,\dj\,$ on the coordinate domain 
$\,\,U\,$ satisfy (\ref{met}). As (iii) implies (i) in Theorem~\ref{meang},
the final clause of the theorem follows, which completes the proof.
\end{proof}

\section{Proofs of Theorem~\ref{noltd} and Corollary~\ref{prdct}}\label{pd}
\setcounter{equation}{0}
It is immediate from the Chris\-tof\-fel-sym\-bol formula that
the leaves of $\,\mathcal{H}$ are totally geodesic if and only if
$\,g_{11}\w,g_{12}\w,g_{22}\w$ depend only on 
$\,x^1$ and $\,x^2\nh$.

Now let $\,\mathcal{V}\hs$ in Theorem~\ref{noltd} be hol\-o\-mor\-phic.
We use the index ranges of (\ref{rng}):
\begin{equation}\label{ind}
\ja,\jb,\jc\in\{1,\hn2\},\qquad
p,q\in\{3,\hn4\}.
\end{equation}
The final clause of Theorem~\ref{lcstr} allows us to assume 
(\ref{met}) -- (\ref{sys}) for suitable local coordinates (\ref{rng}) and 
functions $\,\aj,\bj,\cj,\dj$. Thus, by (\ref{jtt}), whenever
$\,\{\ja,\jb\}=\{1,2\}$,
\begin{equation}\label{jll}
J_{\nnh\ja}^\ja=(-\nnh1)^\ja\nh g_{12}\w,\quad
J_{\nnh\jb}^\ja=(-\nnh1)^\ja\nh g_{\jb\jb}\w,\quad
J_{\nnh p}^\jb=J_{\nnh\ja}^p=0\,\mathrm{\ for\ }\,p\in\{3,4\}.
\end{equation}
With the index ranges (\ref{ind}), we then have
\begin{equation}\label{bez}
B_{\nh\ja p}^{\hs\jb}=0\,\mathrm{\ \ for\ all\ }\,\ja,\jb,\p\mathrm{,\ where\
}\,B_{\nh\ja p}^{\hs\jb}
=J_\jc^{\hs\jb}\vg_{\hskip-2.2pt\ja p}^{\hs\jc}
-J_\ja^{\hs\jc}\vg_{\hskip-2.2pt\jc p}^\jb.
\end{equation}
In fact, at each 
point $\,x$, some real-hol\-o\-mor\-phic vector field
$\,w\,$ on a neighborhood of 
$\,x\,$ is tangent to $\,\mathcal{V}\hs$ and has $\,w\nh_x\w\ne0$. Such
$\,w\,$ obviously realize all nonzero vectors tangent to
$\,\mathcal{V}\hs$ at all points.
Hol\-o\-mor\-phic\-i\-ty of $\,w\,$ amounts to vanishing of the $\,(1,1)$
tensor field $\,S=[J,\nabla\nh w]$. However,
$\,S_{\nh\ja}^{\hs\jb}=B_{\nh\ja p}^{\hs\jb}w^p\nh$,
since $\,w\,$ is tangent to $\,\mathcal{V}\hs$ (and so $\,w\hh^\ja\nh=0$), 
which proves (\ref{bez}).

Next, replacing the Chris\-tof\-fel symbols and components of $\,J\,$ with the
expressions provided by (\ref{olp}-b) -- (\ref{olp}-c) and (\ref{jll}),
then combining two out of the four resulting terms with the aid of
(\ref{det}-i), while noting that the other two cancel each other, and,
finally, using (\ref{olp}-d), we see that,
if $\,\ja\in\{1,2\}\,$ and $\,p\in\{3,4\}$, 
with no summation over the repeated index $\,\ja$,
\begin{equation}\label{bpq}
B_{\nh\ja p}^{\hs\ja}
=(-\nnh1)^\ja\nh g_{11}\w g_{22}\w\aj_p\w.
\end{equation}
Hol\-o\-mor\-phic\-i\-ty of $\,\mathcal{V}\hs$ thus implies, via
(\ref{bez}), (\ref{bpq}), (\ref{sys}-a) and (\ref{met}), that
$\,\mathcal{H}\,$ has totally geodesic
leaves or, in other words, $\,g_{11}\w,g_{12}\w,g_{22}\w$ are functions of
$\,x^1$ and $\,x^2\nh$.

Conversely, let $\,g_{11}\w,g_{12}\w,g_{22}\w$ in (\ref{met})
be functions of
$\,x^1\nh,x^2\nh$, so that they define a metric on an oriented surface 
$\,\varSigma\,$ with the
associated complex structure given by
$\,J_{\nnh\ja}^{\hs\ja}=(-\nnh1)^\ja\nh g_{12}\w$ and
$\,J_{\nnh\ja}^{\hs\jb}=(-\nnh1)^\jb\nh g_{\ja\ja}\w$ 
if $\,\{\ja,\jb\}=\{1,2\}$.
Choosing, locally on this surface, a hol\-o\-mor\-phic function
$\,f$ without critical points, and extending it to our four-man\-i\-fold
$\,M\,$ 
with local coordinates $\,x^1\nh,\dots,x^4\nh$, so as to make it independent
of $\,x^3$ and $\,x^4\nh$, we obtain a hol\-o\-mor\-phic function on the 
complex surface $\,M\nh$. Namely, since $\,J_\ja^p\nh=J_p^\ja\nh=0\,$
in (\ref{jtt}), and
$\,\partial\hh_3\w f\nh=\partial\hn_4\w f\nh=0$, the hol\-o\-mor\-phic\-i\-ty
condition $\,(d\hskip-.8ptf)J=i\hs d\hskip-.8ptf\nnh$, assumed on
$\,\varSigma$, 
remains valid on $\,M\nh$. 
Remark~\ref{holsb} now implies that $\,\mathcal{V}$ is hol\-o\-mor\-phic.
This completes the proof of Theorem~\ref{noltd}.

Corollary~\ref{prdct} is now obvious from the first paragraph of this section.
\begin{rem}\label{smple}To construct examples illustrating 
Theorem~\ref{noltd} it is sufficient -- as well as locally necessary, 
due to Theorem~\ref{lcstr} and Theorem~\ref{meang}(v) -- to
proceed as follows. One
chooses functions $\,\aj,\bj,\cj,\dj\,$ of the local coordinates (\ref{rng}),
so as to satisfy two requirements. First, $\,\aj,\bj\,$ should 
depend only
on $\,(x^1\nnh,x^2)$, and so the initial parts of (\ref{met}) -- 
(\ref{jtt}) define both a metric on an oriented surface $\,\varSigma\,$
with the local coordinates $\,(x^1\nnh,x^2)$, and the associated complex
structure. Secondly, $\,\cj+i\dj$ should form a family,
pa\-ram\-e\-triz\-ed by $\,(x^3\nnh,x^4)$, of hol\-o\-mor\-phic functions on
$\,\varSigma$.
\end{rem}

\renewcommand{\thesection}{\Roman{section}}
\section*{Appendix I: Car\-tan's test}
\setcounter{section}{1}
\setcounter{equation}{0}
\setcounter{thm}{0}
The terminology used in the appendices follows
\cite{bryant-chern-gardner-goldschmidt-griffiths}, and so does the notation,
which makes the conventions adopted here inconsistent, at times, with those
of the preceding sections. The resulting inconvenience is, in our view, minor
compared to the benefit to the reader who wishes to look up our references to
\cite{bryant-chern-gardner-goldschmidt-griffiths}, for further details, and as
a result is able to do that without the need of a ``vocabulary.''

An {\it exterior differential system\/} on a manifold $\,M\,$ is an 
ideal $\,\mathcal{I}\hs$ in the graded algebra $\,\varOmega\hh^*\nnh\nh M\,$ 
of smooth differential forms on $\,M\nh$, closed 
under exterior differentiation. The {\it integral elements\/} (or, {\it 
integral manifolds\/}) of $\,\mathcal{I}\hs$ are the sub\-spaces of
tangent spaces of $\,M\,$ (or, the sub\-man\-i\-folds of $\,M$) on which
all forms in $\,\mathcal{I}\hs$ vanish 
\cite[pp.\,16,\,65]{bryant-chern-gardner-goldschmidt-griffiths}. 
If $\,E\subseteq\tzm\,$ is a $\,p\hs$-di\-men\-sion\-al integral element of
$\,\mathcal{I}\nh$, one sets 
\begin{equation}\label{hee}
H\nh(E)\,=\,\{v\in\tzm:\zeta(v,e_1\w,\dots,e_p\w)\,=\,0\,\mathrm{\ for\ all\
}\,\zeta\in\mathcal{I}\cap\varOmega\hs^{p\hs+1}\nnh M\}\hh.
\end{equation}
See \cite[pp.\,67-68]{bryant-chern-gardner-goldschmidt-griffiths}. The
definition uses a basis $\,e_1\w,\dots,e_p\w$ of $\,E\nh$, but does not
depend on it, since $\,H\nh(E)\,$ is clearly the vector sub\-space of
$\,\tzm\,$ given by 
\begin{equation}\label{hsp}
H\nh(E)\,=\,\{v\in\tzm\nnh:\mathrm{span}(v,E)\,\mathrm{\ is\ an\
integral\ element\ of\ }\,\mathcal{I}\}\hh.
\end{equation}
The set $\,V\hskip-4pt_n\w(\mathcal{I})\,$ of all $\,n$-di\-men\-sion\-al
integral elements of $\,\mathcal{I}\hs$ is contained in the total space of
the Grass\-mann\-i\-an bundle $\,G\nh_n\w(T\nh M)$. Slightly paraphrasing 
a part of
\cite[p.\,74,\,Theorem 1.11]{bryant-chern-gardner-goldschmidt-griffiths}, 
we may state it as follows.

\ \ 

\noindent{\bf Car\-tan's test.\ }{\it Let an exterior differential system\/
$\,\mathcal{I}\hs$ on a manifold $\,M\,$ contain no nonzero forms of
degree\/ $\,0$. If\/ 
$\,\{0\}_z\w\nh=E\nh_0\w\subseteq E\nh_1\w\subseteq E\nh_2\w\subseteq\ldots
\subseteq E\nh_n\w\subseteq\tzm\,$ is  a flag of integral elements of\/
$\,\mathcal{I}\nh$, each $\,E\nnh_k\w$ has the dimension\/ $\,k\,$ and\/
$\,H\nh(E\nnh_k\w)\,$ the co\-dimen\-sion $\,c_k\w$ in\/ 
$\,\tzm\nh$, while
for some neighborhood\/ $\,U\hn$ of\/ $\,E\nh_n\w$ in\/ $\,G\nh_n\w(T\nh M)\,$
the intersection\/ $\,V\hskip-4pt_n\w(\mathcal{I})\cap U\hs$ 
is a smooth manifold of co\-dimen\-sion $\,c_0\w\nnh+\ldots+c_{n-1}\w$ in\/
$\,U\nnh$, then\/ $\,E\nh_n\w$ is ordinary in the sense of\/ 
{\rm\cite[p.\,73,\,Defn.\,1.9]{bryant-chern-gardner-goldschmidt-griffiths}}
and, consequently, according to\/ 
{\rm\cite[p.\,76,\,Cor\nh.\,2.3]{bryant-chern-gardner-goldschmidt-griffiths}},
when\/ $\,\mathcal{I}\hs$ is also assumed real-an\-a\-lyt\-ic, 
there exists an integral manifold of\/ $\,\mathcal{I}\hs$ passing through\/
$\,z\,$ and having at\/ $\,z\,$ the tangent space\/} $\,E\nh_n\w$.
\begin{rem}\label{horiz}In our application of Car\-tan's test (see the
next appendix) $\,M\,$ will be a bundle over some base manifold $\,X\,$
of dimension $\,n$. The {\it vertical distribution\/} on $\,M\,$ is then
tangent to the fibres (that is, levels of the bundle projection
$\,\pi:M\hn\to X$), and we will only be interested in integral elements or 
integral  manifolds of $\,\mathcal{I}\hs$ which are {\it horizontal\/} in the
sense of having a trivial intersection with the vertical space at
the point in question (or, respectively, of this being the case for their 
tangent spaces at all points). Horizontality is a special case of what 
\cite[p.\,103]{bryant-chern-gardner-goldschmidt-griffiths} 
calls an {\it independence condition}.

Horizontal integral manifolds of $\,\mathcal{I}\hs$ are, locally, the same as
graphs of sections $\,X\nh\to M\,$ of the bundle $\,M\,$ over $\,X$, and so
their existence amounts to local solvability of a specific system of partial 
differential equations.
\end{rem}

\section*{Appendix II: Solvability of the system \hs{\rm(\ref{sys})}}
\setcounter{section}{2}
\setcounter{equation}{0}
\setcounter{thm}{0}
To derive Proposition~\ref{slvbl} from Car\-tan's test, we let 
$\,M\,$ be the set of all
\begin{equation}\label{vgt}
(x,y,u,v,\aj,\bj,\cj,\dj,A,B,C,D,E,F\nh,G,H,K,L,P,Q)
\end{equation}
in $\,\bbR\hn^{20}$ having $\,\aj,\cj\in(-\pi/2,\pi/2)$, and treat
$\,M\,$ as a bundle over 
$\,\bbR\hn^4$ with the projection sending (\ref{vgt}) to $\,(x,y,u,v)$.
On $\,M\,$ one has the exterior differential system $\,\mathcal{I}\hs$ 
generated by the following four $\,1$-forms and four $\,2$-forms:
\begin{equation}\label{fof}
\begin{array}{l}
d\aj-A\,dx-B\,dy-C\,du-D\,dv,\\
d\bj-K\,dx-L\,dy+(De^{-\dj}\nh\sec\cj-C\tan\cj)\,du\\
\hskip100pt+\,(D\tan\cj-Ce^\dj\nh\sec\cj)\,dv,\\
d\cj-E\,dx-F\hs dy-G\,du-H\,dv,\\
d\dj+(Fe^{-\bj}\nh\sec\aj-E\tan\aj)\,dx-P\,du-Q\,dv\\
\hskip100pt+\,(F\tan\aj-Ee^\bj\nh\sec\aj)\,dy,\\
dA\wedge dx+dB\wedge dy+dC\wedge du+dD\wedge dv,\\
dK\wedge dx+dL\wedge  dy-(e^{-\dj}\nh\sec\cj\,dD-\tan\cj\,dC)\wedge du\\
\hskip100pt-\,(\tan\cj\,dD-e^\dj\nh\sec\cj\,dC)\wedge dv\,+\,\ldots,\\
dE\wedge dx+dF\wedge dy+dG\wedge du+dH\wedge dv,\\
dP\wedge du+dQ\wedge dv+(\tan\aj\,dE-e^{-\bj}\nh\sec\aj\,dF)\wedge dx\\
\hskip100pt-\,(\tan\aj\,dF-e^\bj\nh\sec\aj\,dE)\wedge dy\,+\,\ldots.
\end{array}
\end{equation}
Here $\,\ldots\,$ stand for two specific functional combinations of the six
exterior products $\,dx\wedge dy,\,dx\wedge du,\,dx\wedge dv,\,dy\wedge du,
\,dy\wedge dv,\,du\wedge dv$, and the
$\,2$-forms in (\ref{fof}) are modified versions of
the opposites of the exterior derivatives of the $\,1$-forms. 
The modifications are based on a simple rule: 
any functional 
multiple of $\,d\aj,d\bj,d\cj$ or $\,d\dj\,$ occurring in an exterior
derivative is replaced by what it would become if the corresponding $\,1$-form
in (\ref{fof}) were declared equal to zero. The ideal $\,\mathcal{I}\hs$ thus
coincides with the one generated by the four $\,1$-forms in (\ref{fof}) and
their exterior derivatives, as each of the above replacements consists 
in adding to a generator of the latter ideal
a functional multiple of another generator.

It is clear that, locally, four-di\-men\-sion\-al integral manifolds of
$\,\mathcal{I}\hs$ which are horizontal in the sense of Remark~\ref{horiz}
are nothing else than graphs of solutions $\,(\aj,\bj,\cj,\dj)\,$ to
(\ref{sys}), with $\,(x^1\nh,x^2\nh,x^3\nh,x^4)=(x,y,u,v)$.

We will apply Car\-tan's test, phrased as in Appendix I,
by selecting, at a fixed point $\,z\in M\,$ with the projection
$\,\mathbf{p}=(x,y,u,v)\in\rfo\nnh$,
a horizontal integral flag $\,E\nnh_k\w$, $\,k=0,1,2,3,4$,
cf.\ Remark~\ref{horiz}, with each $\,E\nnh_k\w$ of dimension $\,k\,$ and 
$\,H\nh(E\nnh_k\w)$ having the co\-dimen\-sions
\begin{equation}\label{czf}
(c_0\w,c_1\w,c_2\w,c_3\w)\,=\,(4,\hs8,\hs12,\hs16)\mathrm{,\ with\
}\,c_0\w\nnh+c_1\w\nnh+c_2\w\nnh+c_3\w\nh=\,40.
\end{equation}
The simultaneous kernel of the four linearly independent $\,1$-forms in 
(\ref{fof}) is a co\-dimen\-sion-four distribution $\,\mathcal{D}$ on
$\,M\nh$. Clearly, $\,E\nh_0\w=\{0\}$, and 
$\,H\nh(E\nh_0\w)=\mathcal{D}\nh_z\w$ has the required co\-dimen\-sion 
$\,c_0\w=4$, while $\,E\nnh_k\w$ must all be contained in
$\,\mathcal{D}\nh_z\w$. Using the dotted
versions of the coordinates (\ref{vgt}) to represent vectors tangent to
$\,M\,$ at a point (\ref{vgt}), we will from now on restrict our
consideration to vectors having
\begin{equation}\label{rqt}
\begin{array}{l}
\dot\aj=A\dot x+B\dot y+C\dot u+D\dot v,\\
\dot\bj=K\dot x+L\dot y-(De^{-\dj}\nh\sec\cj-C\tan\cj)\dot u
-(D\tan\cj-Ce^\dj\nh\sec\cj)\dot v,\\
\dot\cj=E\dot x+F\dot y+G\dot u+H\dot v,\\
\dot\dj=P\dot u+Q\dot v-(Fe^{-\bj}\nh\sec\aj-E\tan\aj)\dot x
-(F\tan\aj-Ee^\bj\nh\sec\aj)\dot y,
\end{array}
\end{equation}
which amounts to requiring that they lie within the distribution
$\,\mathcal{D}$.
With any given $\,i,j\in\{1,2,3,4\}$, for a pair 
$\,(\dot x_i\w,\dot y_i\w,\dots,\dot P\nh\nnh_i\w,\dot Q_i\w),\,\, 
(\dot x\nh_j\w,\dot y\hn_j\w,\dots,\dot P\nnh\nnh_j\w,\dot Q_j\w)\,$ of vectors
with (\ref{rqt}), being annihilated by the $\,2$-forms in (\ref{fof}) means
precisely that
\begin{equation}\label{prc}
\begin{array}{l}
\dot x\hn_j\w \dot A_i\w-\dot x_i\w\dot A_j\w
+\dot y_j\w\dot B_i\w-\dot y_i\w\dot B_j\w
+\dot u_j\w \dot C_i\w-\dot u_i\w\dot C\nnh_j\w
+\dot v_j\w\dot D\nh_i\w-\dot v_i\w\dot D\nh_j\w\,=\,0,\\
\dot x\nh_j\w\dot K_i\w\nnh-\nh\dot x_i\w\dot K\nh_j\w\nnh
+\nh\dot y\nh_j\w\dot L_i\w-\dot y_i\w\dot L_j\w\nnh
\nh+(\tan\cj)(\dot u_j\w\dot C\hn_i\w\nnh-\nh\dot u_i\w\dot C\nh_j\w)\nh
-\nh(\tan\cj)(\dot v\nh_j\w\dot D\nh_i\w\nnh-\nh\dot v_i\w\dot D\nnh_j\w)
\hskip-10pt\\
\hskip40pt+\,e^\dj\nh(\sec\cj)(\dot v_j\w\dot C_i\w-\dot v_i\w\dot C\nnh_j\w)
-e^{-\dj}\nh(\sec\cj)(\dot u_j\w\dot D\nh_i\w
-\dot u_i\w\dot D\nh_j\w)\,
=\,\ldots\hs,\\
\dot x\hn_j\w\dot E\hn_i\w-\dot x_i\w\dot E\nh_j\w
+\dot y_j\w\dot F\nnh\nh_i\w-\dot y_i\w\dot F\nnh\nnh_j\w
+\dot u_j\w \dot G_i\w-\dot u_i\w\dot G_j\w
+\dot v_j\w\dot H_i\w-\dot v_i\w\dot H_j\w\,=\,0,\\
\dot u\hn_j\w\dot P\nnh\hn_i\w\nnh-\nh\dot u_i\w\dot P\nnh\nh_j\w\nnh
+\nh\dot v\nh_j\w\dot Q_i\w\nnh-\nh\dot v_i\w\dot Q\hn_j\w\nnh
+(\tan\aj)(\dot x\nh_j\w\dot E\hn_i\w\nnh-\nh\dot x_i\w\dot E\nh_j\w)\nh
-\nh(\tan\aj)(\dot y_j\w\dot F\nnh\nh_i\w\nnh
-\nh\dot y_i\w\dot F\nnh\nnh_j\w)\hskip-10pt\\
\hskip40pt+\,e^\bj\nh(\sec\aj)(\dot y_j\w\dot E\hn_i\w
-\dot y_i\w\dot E\nh_j\w)
-e^{-\bj}\nh(\sec\aj)(\dot x_j\w\dot F\nnh\nh_i\w
-\dot x_i\w\dot F\nnh\nnh_j\w)\,=\,\ldots\hs,
\end{array}
\end{equation}
where $\,\ldots\,$ now denotes two specific functional combinations of 
the six expressions $\,\dot x_i\w\dot y_j\w-\dot x\hn_j\w\dot y_i\w,\dots,
\dot u_i\w\dot v_j\w-\dot u_j\w\dot v_i\w$. 
On a purely formal level, (\ref{prc}) is the result of replacing, in (\ref{fof}), 
$\,dA\wedge dx\,$ with
$\,\dot A_i\w\dot x\hn_j\w-\dot A_j\w\dot x_i\w$ (and analogously for the
remaining eleven capital letters $\,B,C,\dots,P,Q\,$ and three low\-er-case
letters $\,y,u,v$), with a similar replacement applied to the right-hand sides.
For fixed $\,\dot x_i\w,\dot x\hn_j\w,\dot y_i\w,\dot y_j\w\dots,
\dot u_i\w,\dot u_j\w,\dot v_i\w,\dot v_j\w$, with all choices of $\,i,j$,
(\ref{prc}) are linear homogeneous in the forty-eight
dot\-ted-and-sub\-script\-ed versions of our twelve capital letters.

To determine the co\-dimen\-sions $\,c_k\w$ of $\,H\nh(E\nnh_k\w)$, for
suitably selected horizontal integral elements $\,E\nnh_k\w$ of
dimensions $\,k=1,2,3$, 
we use a unified approach: fixing a basis of $\,E\nnh_k\w$, assumed to have
some ``generic'' property, and finding out
how many independent conditions are involved in extending this basis by
one vector, so that the span will still be an integral element. The basis always
consists of vectors
$\,(\dot x\nh_j\w,\dot y\hn_j\w,\dots,\dot P\nnh\nnh_j\w,\dot Q_j\w)$ 
with (\ref{rqt}), $\,1\le j\le k$, with 
the projections 
$\,(\dot x\nh_j\w,\dot y\hn_j\w,\dot u_j\w,\dot v_j\w)$ linearly independent
in $\,\rfo$ (due to horizontality), while equations 
(\ref{prc}) hold for all $\,i,j\,$ with
$\,1\le i<j\le k$, since $\,E\nnh_k\w$ is an integral element. 
A new vector $\,(\dot x_i\w,\dot y_i\w,\dots,\dot P\nh\nnh_i\w,\dot Q_i\w)\,$ 
added to this basis so as to span an integral element is thus subject
to precisely $\,4k$ equations (\ref{prc}), where $\,j\,$ ranges from
$\,1\,$ to $\,k\,$ and $\,i>k$ is fixed. 
Using the notation
$\,\dot{\mathbf{x}}\hn_j\w=(\dot x\nh_j\w,\dot y\hn_j\w)\,$ and 
$\,\dot{\mathbf{u}}\hn_j\w=(\dot u_j\w,\dot v_j\w)$, we now represent
each group of four out of the $\,4k$ linear homogeneous equations, with a
single value of $\,j$ and the sixteen unknowns 
$\,\dot A_i\w,\dot B_i\w,\dot C_i\w,\dot D_i\w,\dot E\hn_i\w,\dot F\nnh\nh_i\w,
\dot G\hn_i\w,\dot H\nh_i\w,\dot K_i\w,\dot L_i\w,\dot P\nnh\nh_i\w,\dot Q_i\w,
\dot x_i\w,\dot y_i\w,\dot u_i\w,\dot v_i\w$, by the $\,4\times\nnh16\,$ matrix
\begin{equation}\label{rpm}
\left[\begin{matrix}
\dot{\mathbf{x}}\hn_j\w&\dot{\mathbf{u}}\hn_j\w&\mathbf0&\mathbf0&\mathbf0&\mathbf0
&*&*\cr
\mathbf0&\Phi\dot{\mathbf{u}}\hn_j\w&\mathbf0&\mathbf0&\dot{\mathbf{x}}\hn_j\w
&\mathbf0&*&*\cr
\mathbf0&\mathbf0&\dot{\mathbf{x}}\hn_j\w&\dot{\mathbf{u}}\hn_j\w&\mathbf0&\mathbf0
&*&*\cr
\mathbf0&\mathbf0&\Psi\dot{\mathbf{x}}\hn_j\w&\mathbf0&\mathbf0
&\dot{\mathbf{u}}\hn_j&*&*\end{matrix}\right]\nnh\nnh,
\end{equation}
where every boldface vector entry stands for a row-like 
pair of scalar entries, including 
$\,\mathbf0=(0,0)$, while
the linear en\-do\-mor\-phisms $\,\Phi,\Psi\,$ of
$\,\rto$ are given by
\begin{equation}\label{fps}
\begin{array}{l}
\Phi(u,v)\,=\,(u\tan\cj\,+\,v\hs e^\dj\nh\sec\cj,\,-ue^{-\dj}\nh\sec\cj\,
-\,v\tan\cj)\,\mathrm{\hs\ \ and}\\
\Psi(x,y)\,=\,(x\tan\aj\,+\,y\hs e^\bj\nh\sec\aj,\,-xe^{-\bj}\nh\sec\aj\,
-\,y\tan\aj)\mathrm{,\nh\ \ so}\\
\mathrm{that\ }\hn\Phi\nh,\nnh\Psi\nh\mathrm{\ are\ linear\ auto\-mor\-phisms\
with\ no\ real\ eigen\-values.}
\end{array}
\end{equation}
In fact, $\,\Phi\,$ and $\,\Psi\,$ are complex structures, cf.\ 
(\ref{met}) -- (\ref{jtt}). The eight asterisks in (\ref{rpm}) are row-like 
pairs of entries, which
we choose not to identify, as they have no bearing on our discussion.

Stacking $\,k$, that is, one, or two, or three, matrices (\ref{rpm}) on top of 
one another, for $\,k\,$ different values of $\,j$, we obtain 
a $\,4k\times\nh16\,$ matrix, which,
\begin{equation}\label{gnc}
\mathrm{with\ a\ generic\ choice\  of\ }\,E\nnh_k\w\mathrm{,\ will\ have\
linearly\ independent\ rows.}
\end{equation}
Here are the details. 
For $\,k=1\,$ our genericity requirement reads
$\,\dot{\mathbf{x}}\hn_j\w\ne\mathbf0\ne\dot{\mathbf{u}}\hn_j\w$. In other
words, we want the line $\,d\pi\nh_z\w(E\nh_1\w)\subseteq\rfo$ not
to be contained in either of two specific coordinate planes. If a linear
combination of the rows of (\ref{rpm}) vanishes, the coefficient of the first
(or second, or third, or fourth) row equals zero, as shown by the first,
fifth, fourth, and sixth of the eight ``columns'' of (\ref{rpm}).

If $\,k=2$, we impose the following genericity condition: both
pairs $\,\dot{\mathbf{x}}\hn_1\w,\dot{\mathbf{x}}\hn_2\w$ and
$\,\dot{\mathbf{u}}\hn_1\w,\dot{\mathbf{u}}\hn_2\w$ should be
linearly independent in $\,\rto\nh$. Equivalently, the plane
$\,d\pi\nh_z\w(E\nh_2\w)\,$ in $\,\rfo$ projects sur\-ject\-i\-ve\-ly 
onto two specific coordinate planes (that is, intersects both of them
trivially). Placing the matrix (\ref{rpm}) with $\,j=1\,$ on top of its
version for $\,j=2$, we obtain a matrix with eight linearly independent rows. 
In fact, as in the preceding paragraph, a vanishing linear combination of the
rows must have zero coefficients of the rows $\,1\,$ and $\,5$, or $\,2\,$ and
$\,6$, or $\,3\,$ and $\,7$, or $\,4\,$ and $\,8$, as one sees looking, again, 
at the first, fifth, fourth, and sixth ``columns.''

Finally, let $\,k=3$. This time we require both triples
$\,\dot{\mathbf{x}}\hn_1\w,\dot{\mathbf{x}}\hn_2\w,\dot{\mathbf{x}}\hn_3\w$ and
$\,\dot{\mathbf{u}}\hn_1\w,\dot{\mathbf{u}}\hn_2\w,\dot{\mathbf{u}}\hn_3\w$ to be
in general position, in the sense that any two vectors from the triple
are linearly independent in $\,\rto\nh$. Triples $\,(q_1\w,q_2\w,q_3\w)\in\rtr$
with $\,q_1\w\dot{\mathbf{x}}\hn_1\w+q_2\w\dot{\mathbf{x}}\hn_2\w
+q_3\w\dot{\mathbf{x}}\hn_3\w=\mathbf0\,$ thus form {\it a one-di\-men\-sion\-al 
sub\-space\/} of $\,\rtr\nh$, and similarly for the vectors
$\,\dot{\mathbf{u}}\hn_j\w$. A further ge\-ner\-ic-type assumption needed here is
that this
\begin{equation}\label{ods}
\begin{array}{l}
\mathrm{one}\hyp\mathrm{di\-men\-sion\-al\ sub\-space\ of\
}\hs\rtr\nnh\mathrm{\ associated\ with}\\
\dot{\mathbf{u}}\hn_1\w,\dot{\mathbf{u}}\hn_2\w,\dot{\mathbf{u}}\hn_3\w\mathrm{\hs\
is\ different\ from\ the\ one\ for\
}\,\dot{\mathbf{x}}\hn_1\w,\dot{\mathbf{x}}\hn_2\w,\dot{\mathbf{x}}\hn_3\w.
\end{array}
\end{equation}
Stacking together, as before, three copies of
(\ref{rpm}), with different values of $\,j$, suppose now that we have a
vanishing linear combination of the resulting twelve rows, with some coefficients 
$\,q_1\w,r_1\w,s_1\w,t_1\w,q_2\w,r_2\w,s_2\w,t_2\w,q_3\w,r_3\w,s_3\w,t_3\w$.
Looking at the first, fifth, fourth and sixth ``columns'' we see, writing
$\,\mathbf{q}\,$ for $\,(q_1\w,q_2\w,q_3\w)$, etc., 
that both pairs $\,\mathbf{q},\mathbf{r}\,$ and $\,\mathbf{s},\mathbf{t}\,$ are
linearly dependent in $\,\rtr\nh$.

Next, $\,\mathbf{q}=\mathbf0\,$ both when 
$\,\mathbf{r}=\mathbf0\,$ and when $\,\mathbf{r}\,$ is nonzero:
we use the first two ``columns'' and -- in the former case -- invoke (\ref{ods})
while, in the latter, write $\,\mathbf{q}=\ly\mathbf{r}\,$ with a scalar $\,\ly$,
so that the coefficients $\,r_1\w,r_2\w,r_3\w$ produce vanishing linear 
combinations of the vectors 
$\,\ly\dot{\mathbf{x}}\hn_j\w$, as well as of
$\,\Phi\dot{\mathbf{u}}\hn_j\w+\ly\dot{\mathbf{u}}\hn_j\w$ (and hence,
due to the last line of (\ref{fps}), also of $\,\dot{\mathbf{u}}\hn_j\w$). Then 
(\ref{ods}) gives $\,\ly=0$.

Very similarly, $\,\mathbf{s}=\mathbf0$, whether or not $\,\mathbf{t}=\mathbf0$,
as one sees using the third and fourth ``columns'' instead of the first two. 
With $\,\mathbf{q}=\mathbf{s}=\mathbf0$, the second and fifth (or, third
and sixth) ``columns'' show that, with the coefficients $\,r_1\w,r_2\w,r_3\w$ 
(or, $\,t_1\w,t_2\w,t_3\w$), the resulting linear combinations of the vectors
$\,\Phi\dot{\mathbf{u}}\hn_j\w$ and $\,\dot{\mathbf{x}}\hn_j\w$ (or,
$\,\Psi\dot{\mathbf{x}}\hn_j\w$ and $\,\dot{\mathbf{u}}\hn_j\w$) are both
equal to zero. Now (\ref{ods}), combined with the last line of (\ref{fps}),
implies that $\,\mathbf{r}=\mathbf{t}=\mathbf0$, proving our claim
(\ref{gnc}) about linear independence of the rows.

To choose $\,E\nnh_k\w$ of dimensions $\,k=1,2,3$, each 
of them generic and contained in the next one, we fix $\,E\hn_3\w$ with a
generic basis as described above, and declare $\,E\nnh_k\w$, for $\,k=1,2$, to
be the span of the first $\,k\,$ vectors of this basis, which obviously makes
them generic. For each $\,k\in\{1,2,3\}$, (\ref{gnc}) guarantees 
that the solution space, forming the kernel of a rank $\,4k\,$ operator
$\,\bbR\hn^{16}\nh\to\bbR^{4k}\nnh$,\hs\ is of dimension $\,16-4k$, and hence of
the co\-dimen\-sion $\,c_k\w=4(k+1)\,$ in $\,\bbR\hn^{20}\nh$, 
as required in (\ref{czf}).

As a final step we will now show that, within the af\-fine space of all 
four-di\-men\-sion\-al horizontal sub\-spaces of $\,\bbR\hn^{20}\nnh$, cf.\ 
Remark~\ref{graff}, the integral elements $\,E\nnh_4\w$ of $\,\mathcal{I}$ at
$\,z\,$ form a co\-dimen\-sion-forty af\-fine sub\-space. Due to
(\ref{czf}), this provides the correct co\-dimen\-sion of 
$\,V\hskip-4pt_4\w(\mathcal{I})\cap U\hs$ in Car\-tan's test (Appendix I), with
the manifold structure of an af\-fine bundle (the fibres having the same 
dimension $\,24\,$ at all points).

Every horizontal four-di\-men\-sion\-al sub\-space $\,E\nnh_4\w$ of
$\,\mathcal{D}\nh_z\w\subseteq\bbR\hn^{20}$ has a unique basis of 
vectors $\,(\dot x\nh_j\w,\dot y\hn_j\w,\dots,\dot P\nnh\nnh_j\w,\dot Q_j\w)$ 
with (\ref{rqt}), $\,1\le j\le 4$, such that the projections 
$\,(\dot x\nh_j\w,\dot y\hn_j\w,\dot u_j\w,\dot v_j\w)\,$ form the 
standard basis of $\,\rfo\nh$. Requiring $\,E\nnh_4\w$ to be an  integral element
of $\,\mathcal{I}\hs$ amounts to imposing,
on the forty-eight unknowns
\[
\dot A_j\w,\dot B_j\w,\dot C\nnh_j\w,\dot D\hn_j\w,\dot E\nh_j\w,\dot F\nnh\nnh_j\w,
\dot G\nh_j\w,\dot H\nnh_j\w,\dot K_j\w,\dot L_j\w,\dot P\nnh\nnh_j\w,\dot Q_j\w,
\quad j=1,2,3,4,
\]
the system of twen\-ty-four (generally nonhomogeneous) linear
equations (\ref{prc}), for the six pairs $\,(i,j)\,$ with $\,1\le i<j\le 4$.

{\it The\/ $\,24\times\nh48\,$ matrix of this system is of rank twen\-ty-four}.
Namely, replacing all the right-hand sides by zero, we obtain a system stating
that the four $\,2$-forms in (\ref{fof}) (with the dots $\,\ldots\,$ treated as
zero) vanish on $\,E\nnh_4\w$. According to 
Car\-tan's lemma (\ref{crt}) this means precisely two
things.

First, the matrix with the four rows
$\,(\dot A_j\w,\dot B_j\w,\dot C\nnh_j\w,\dot D\hn_j\w)$, $\,j=1,2,3,4$, is
symmetric, and so is the one with the rows 
$\,(\dot E\nh_j\w,\dot F\nnh\nnh_j\w,\dot G\nh_j\w,\dot H\nnh_j\w)$,
resulting in $\,10+10=20$ parameters, for now free (but subjected to two
constraints below).

Secondly, two more symmetric $\,4\times\nh4\,$ matrices similarly
arise from the second and, respectively, fourth 
$\,2$-forms in (\ref{fof}). However the last two (or, respectively, first two)
of their columns are already expressed in terms of the two symmetric matrices
mentioned earlier. The two new symmetry conditions thus amount to introducing
six new free parameters: $\,\dot K\hn_1\w,\dot L_2\w,\dot K\hn_2\w=\dot L_1\w,
\dot P\nnh\nh_3\w,\dot Q\hn_4\w,\dot P\nnh\nnh_4\w=\dot Q_3\w$, while also subjecting the previous
ones to two constraints:
\[
(e^\bj\nh\dot E\nh_1\w+e^{-\bj}\nnh\dot F\nnh\nnh_2\w)\sec\aj
-2\dot E\nh_2\w\tan\aj\,
=\,(e^\dj\hn\dot C\nh_3\w+e^{-\dj}\nnh\dot D\nh_4\w)\sec\cj
-2\hh \dot C\nh_4\w\tan\cj\,=\,0,
\]
leading, in the end, to $\,20+6-2=24\,$ free parameters.

Since Car\-tan's lemma (\ref{crt})
has the form of an equivalence, the solution space of the
corresponding homogeneous system is of dimension twen\-ty-four, which proves
the italicized claim made at the beginning of one of the earlier paragraphs.

The horizontal integral elements $\,E\nnh_4\w$ thus form an af\-fine space of
dimension twen\-ty-four, and hence of the correct co\-dimen\-sion forty 
required by Car\-tan's test.

\ \

\ \

\noindent{\bf Funding.}  
This work was supported by the National Research Foundation of Korea (NRF)
grant funded by the Korea government (MSIT) (RS-2024-00334956).

\end{document}